\documentclass[11pt]{article}

\usepackage{amsmath,amssymb,amsthm,yhmath}
\usepackage{verbatim}
\usepackage[french]{babel}

\usepackage[utf8]{inputenc}
\usepackage[active]{srcltx}
\author{Olivier Garet}

\newcommand{\N}{\ensuremath{\mathbb{N}}}

\newcommand{\Z}{\ensuremath{\mathbb{Z}}}
\newcommand{\E}{\ensuremath{\mathbb{E}}}
\newcommand{\Q}{\ensuremath{\mathbb{Q}}}
\renewcommand{\P}{\ensuremath{\mathbb{P}}}
\newcommand{\R}{\ensuremath{\mathbb{R}}}
\newcommand{\Rd}{\ensuremath{\mathbb{R}^d}}

\newcommand{\C}{\ensuremath{\mathbb{C}}}

\newcommand{\1}{1\hspace{-2.7mm}1}

\newcommand{\miniop}[3]{%
\renewcommand{\arraystretch}{0.6}
\begin{array}{c}
{\scriptstyle #1}\\
#2\\
{\scriptstyle #3}
\end{array}
\renewcommand{\arraystretch}{1}}
\renewcommand\Re{\text{Re}\ }

\renewcommand{\epsilon}{\varepsilon}

\newcommand{\floor}[1]{\lfloor #1 \rfloor}
 \begin{document}

\newtheorem{theo}{\textsc{Théorème}}
\newtheorem{lem}{\textsc{Lemme}}
\newtheorem{coro}{\textsc{Corollaire}}
\newtheorem{rem}{\textsc{Rem}}
\author{Olivier \textsc{Garet}}
\date{article paru dans Quadrature, \no 96, Avril-mai-juin 2015 }
\title{Les lois Zêta pour l'arithmétique}

\maketitle

\begin{abstract}
On (re)visite ici avec un regard probabiliste  un certain nombre de résultats connus de la théorie analytique des nombres.
Au centre de l'article se trouvent les lois Zêta, qui nous sont une consolation de l'inexistence d'une loi uniforme sur $\N$.
Elles nous permettront par exemple d'étudier la densité naturelle des couples d'entiers ou d'entiers de Gauss premiers entre eux, ainsi que d'autres problèmes analogues. Au passage, on retrouvera la décomposition de la fonction Zêta de Riemann sous forme d'un produit eulérien et une généralisation aux sommes de fonctions multiplicatives.
\end{abstract}
\section{De jolis résultats}

Un résultat élémentaire bien connu de théorie des nombres, dû à Dirichlet, est le suivant: si
je prends, de manière indépendante, deux nombres choisis uniformément entre $1$ et $n$, alors, la probabilité $P_n$ que ces deux nombres soient premiers entre eux vérifie $\miniop{}{\lim}{n\to +\infty} P_n=\frac{6}{\pi^2}$.
Un résultat un peu moins connu analogue est que la probabilité $Q_n$ qu'un nombre choisi uniformément entre $1$ et $n$ soit sans facteur carré (c'est à dire qu'on ne peut le diviser par le carré d'un entier différent de 1) possède le même comportement.
Voici deux résultats d'essence probabiliste, qui pourtant, ne sont pas tout à fait des grands classiques de la littérature universitaire en probabilités. Ils sont en revanche bien connus en théorie analytique des nombres, et dans ce contexte, les preuves proposées ont une présentation peu probabiliste. Dans le texte qui suit, nous allons ramener ces résultats dans un cadre probabiliste, notamment à l'aide des lois Zêta, qui sont des lois très simples dotées de jolies propriétés arithmétiques.
Un avantage de l'approche probabiliste est que l'on n'a pas besoin d'estimées très fines pour montrer la convergence. Dans une partie finale, on pourra donner ainsi une preuve relativement élémentaire du calcul de la densité asymptotique des couples d'entiers de Gauss premiers entre eux. 
\section{Des preuves détournées}

Les preuves les plus couramment trouvées dans la littérature se divisent souvent en deux étapes:
\begin{itemize}
\item Montrer la convergence de $(P_n)_{n\ge 1}$ (ou $(Q_n)_{n\ge 1}$)
\item Identifier la limite.
\end{itemize}
Pour la première étape, une preuve simple  repose sur la formule du crible:
prenons $X$ et $Y$ deux variables aléatoires indépendantes suivant la loi uniforme sur $\{1,\dots,n\}$. On exprime alors la probabilité cherchée comme le complémentaire de ``un nombre premier divise $X$ et $Y$'' et ``le carré d'un nombre premier divise $X$''. Notant $(p_n)_{n\ge 1}$ la suite des nombres premiers, on optient avec la formule du crible

\begin{align*}
P_n&=\frac1{n^2}\miniop{}{\sum}{B\subset\{1,\dots,n\}} (-1)^{|B|} \left({\floor{\frac{n}{{\prod}_{i\in B}p_i}}}\right)^2\\
&=\frac1{n^2} \miniop{}{\sum}{B\subset\{1,\dots,n\}}    \mu({\prod}_{i\in B}p_i)\left({\floor{\frac{n}{{\prod}_{i\in B}p_i}}}\right)^2,
\end{align*}
où $\mu$ est la fonction de Möbius:  $ \mu(n)=0$ si $n$ est divisible par un carré et $\mu(n)=(-1)^i$ si $n$ s'écrit comme produit de $i$ nombres premiers distincts.
En réindexant les termes et en rajoutant de nombreux termes nuls, on a finalement
 \begin{align*}
P_n&=\frac1{n^2}\sum_{k\ge 1} \mu(k) \left({\floor{\frac{n}{k}}}\right)^2,
\end{align*}
et, de manière similaire
\begin{align*}
Q_n&=\frac1{n}\miniop{}{\sum}{B\subset\{1,\dots,n\}} (-1)^{|B|} {\floor{\frac{n}{{\prod}_{i\in B}p_i^2}}}\\&=\frac1{n}\sum_{k\ge 1} \mu(k) {\floor{\frac{n}{k^2}}}.
\end{align*}

En utilisant un argument de convergence dominée, on obtient alors la convergence de ces deux suites vers $\sum_{n=1}^{+\infty}\frac{\mu(n)}{n^2}$.

Pour référence, mentionnons qu'une autre approche existe, consistant à exprimer $(P_n)$ comme une somme de la fonction indicatrice d'Euler, que l'on exprime elle-même à l'aide de la fonction de Möbius\footnote{plus précisément, grâce à l'identité $\phi=\mu*\text{Id}$}. Cette approche, utilisée par exemple par Tissier~\cite{tissier}, est un cas particulier de la méthode de l'hyperbole de Dirichlet. Pour une description générale de cette technique, nous renvoyons le lecteur à Bordellès~\cite{borde} ou Tenenbaum~\cite{tenenbible}.

Dans tous les cas,  le problème d'identifier la somme de cette série demeure.

La manière la plus efficace est d'utiliser un  produit de Dirichlet: on note que
$$\left(\sum_{n=1}^{+\infty}\frac{\mu(n)}{n^2}\right)\left(\sum_{n=1}^{+\infty}\frac{1}{n^2}\right)= \sum_{n=1}^{+\infty}\frac{(\mu*1)_n}{n^2},$$
où $*$ désigne la convolution arithmétique des suites.
$$(a*b)_n=\sum_{d|n} a_d b_{n/d}.$$
Comme $\mu*1=\delta_1$ (c'est la formule d'inversion de Möbius), on obtient l'identification recherchée.~\footnote{
Faraut et Khalili~\cite{faraut} utilisent une technique d'identification assez similaire, basée cette fois sur l'identité $1*1=\tau$, où $\tau$ est la fonction \og nombre de diviseurs\fg.}
Donnons une brève preuve combinatoire de la formule d'inversion de Möbius: soit $k\ge 1$.
\begin{itemize}
\item Si $k=1$, l'identité $(\mu*1)(k)=\mu(1)=1=\delta_1(k)$ est immédiate.
\item Sinon, si les facteurs premiers sont $q_1,\dots,q_m$, les seuls diviseurs $d$ de $k$ qui apportent une  contribution non nulle à $\sum_{d|n}\mu(d)$ s'écrivent comme produit des éléments d'une partie de $\{q_1,\dots,q_m\}$, de sorte que

$$\sum_{d|n}\mu(d)=\sum_{B\subset\{q_1,\dots,q_m\}}\mu(\prod_{a\in B} a)=\sum_{B\subset\{q_1,\dots,q_m\}} (-1)^{|B|}.$$
L'application $B\mapsto \{q_1\}\cup B$ crée clairement une bijection entre les parties de $\{q_1,\dots,q_m\}$ ne contenant pas $q_1$ et celles la contenant. \`A l'évidence, les deux parties ainsi associées apportent des contributions opposées à la somme totale, qui est donc nulle.
\end{itemize}

On notera que ces preuves sont essentiellement combinatoires, et font toutes l'usage d'au moins deux formules de convolution ou d'inversion: en effet, la formule de Poincaré, d'une part, la formule de Möbius, d'autre part, sont toutes deux des formules d'inversion. Mais si le problème se résout à l'aide de deux inversions, peut-être pourrait-on n'en faire aucune ?
On  est donc tenté de se demander si on ne peut pas trouver une preuve plus directe, ou plus probabiliste.

\section{Des indices}

Une manière simple de deviner la limite est de passer par les lois Zêta.

Pour $s>1$, on appelle loi Zêta (ou loi de Zipf) de paramètre $s$ la loi sur $\N^*$ qui assigne la masse $\frac{n^{-s}}{\zeta(s)}$ au point $n$.
Le coefficient de renormalisation $\zeta(s)$ est la célèbre fonction $\zeta$ de Riemann, définie pour $s>1$ par
$$\zeta(s)=\sum_{k=1}^{+\infty}\frac1{k^s}.$$
\begin{lem}
\label{limzeta}
La fonction $\zeta$ a une limite infinie en $1$.
\end{lem}
\begin{proof}
 Cela se voit aisément par un argument de monotonie: comme la série harmonique diverge, pour tout $A$ on peut trouver $n$ tel que $\sum_{k=1}^n \frac1{k}>A$. Par continuité de la fonction $s\mapsto \sum_{k=1}^n \frac1{k^s}$, il existe $s_0$ tel que  $\sum_{k=1}^n \frac1{k^s}>A$ pour $s\in \mathopen]1,s_0\mathclose[$. Il est alors clair que $\zeta(s)>A$ pour $s\in \mathopen]1,s_0\mathclose[$.
\end{proof}

Les lois Zêta apparaissent naturellement comme consolatrices de l'inexistence d'une loi uniforme sur $\N^*$. Précisément, on a

\begin{lem}
\label{calculproba}
Soit $s>0$.
Si $X$ est une variable aléatoire sur $\N^*$  telle que tout tout $n\ge 1$, la probabilité de l'événement $n|X$ ($n$ divise $X$) vérifie $\P(n|X)=\frac1{n^s}$, alors pour des nombres $n_1,\dots, n_k$ deux à deux premiers entre eux, les événements $\{n_1|X\}$, \dots,  $\{n_k|X\}$ sont globalement (ou mutuellement) indépendants.
\end{lem}
\begin{proof}
Soient $a_1,\dots ,a_r$ des entiers distincts entre $1$ et $k$.
Comme les entiers $n_{a_i}$ sont premiers entre eux, on a
$$\miniop{r}{\cap}{i=1}  \{n_{a_i}| X\}=\{\miniop{r}{\prod}{i=1} n_{a_i} | X\}.$$
Donc
\begin{align*}\P(\miniop{r}{\cap}{i=1}  \{n_{a_i}| X\})&=\P((\miniop{r}{\prod}{i=1} n_{a_i}) | X)\\&= (\miniop{r}{\prod}{i=1} n_{a_i})^{-s}=  \miniop{r}{\prod}{i=1} n_{a_i}^{-s}\\
&=  \miniop{r}{\prod}{i=1} \P(n_{a_i}|X)
\end{align*}
\end{proof}
Pour $s>1$, considérons $X$ suivant la loi Zêta de paramètre $s$. Pour tout entier naturel $n$ non nul, on a
\begin{align*}
\P(n|X)&=\sum_{k=1}^{+\infty}\P(X=kn)= \zeta(s)^{-1}\sum_{k=1}^{+\infty} (kn)^{-s}\\
&= n^{-s}\zeta(s)^{-1}\sum_{k=1}^{+\infty} k^{-s}=n^{-s}.
\end{align*}
Ainsi, la loi  Zêta de paramètre $s$ satisfait la propriété requise par le lemme, que l'on peut donc appliquer:
\begin{align*}
\frac1{\zeta(s)}&=\P(X=1)=\P(\forall n\ge 1; p_n\nmid X)\\ &=\prod_{n\ge 1} \P(p_n\nmid X)=\prod_{n\ge 1}(1-\frac1{p_n^s}).
\end{align*}
L'égalité $$\P(\forall n\ge 1; p_n\nmid X)=\prod_{n\ge 1} \P(p_n\nmid X)$$ mérite quelques mots: d'abord, on utilise le fait que si une famille finie d'événements est indépendante, alors la famille formée par leurs complémentaires l'est aussi. Pour cela, il suffit de montrer que si on remplace un des éléments de la famille par son complémentaire, la famille alors obtenue est encore indépendante, puis d'itérer.~\footnote{Pour les détails, on pourra si nécessaire se reporter à~\cite{gk}, exercice 18 page 50.} 
Pour conclure, il suffit d'utiliser le théorème de continuité séquentielle décroissante: pour toute suite décroissante d'événements, la probabilité de l'intersection est la limite des probabilités.

Nous venons ainsi de donner une preuve probabiliste d'une formule célèbre de la théorie des nombres, initialement dûe à Euler:
\begin{align}
\label{euler}
\forall s>1\quad \zeta(s)^{-1}&= \prod_{n\ge 1}(1-\frac1{p_n^s}).
\end{align}

En passant au logarithme, on obtient pour tout $s>1$,  $$\log \zeta(s)= \sum_{n\ge 1} -\log(1-\frac1{p_n^s})\le \sum_{n\ge 1} -\log(1-\frac1{p_n}).$$
En faisant tendre $s$ vers $1$, on obtient  la divergence de la série de terme général $-\log(1-\frac1{p_n})$ et, par équivalent, de la série des $\frac1{p_n}$
Ainsi, pour $s\in \mathopen]0,1\mathclose]$, s'il existait une variable aléatoire sur $\N^*$ telle  $\P(n|X)=\frac1{n^s}$ pour tout $n\ge 1$, le deuxième lemme de Borel--Cantelli~\footnote{Le 2e lemme de Borel--Cantelli dit que si des événements indépendants $(A_n)_{n\ge 1}$ vérifient $\sum_{n=1}^{+\infty} \P(A_n)=+\infty$, alors la probabilité que l'on voit simultanément la réalisation d'une infinité des événements $A_n$ vaut 1. Voir par exemple Garet et Kurtzmann~\cite{gk}.} assurerait $$\P(p_n|X\text{ pour une infinité de valeurs de }n)=1,$$ ce qui est absurde, car l'événement en question est vide.
En particulier, il n'existe pas de variable aléatoire sur $\N^*$ telle  $\P(n|X)=\frac1{n}$ pour tout $n\ge 1$.

On restera donc avec $s>1$. 
Soit donc $X$ suivant la loi Zêta de paramètre~$s$. En procédant comme précédemment, on a
\begin{align*}
\P(X\text{ sans facteur carré})&=\P(\forall i\ge 1;\quad p_i^2\ \text{ne divise pas }X)\\&=\prod_{i\ge 1}(1-\frac1{p_i^{2s}})=\frac1{\zeta(2s)},
\end{align*}
où la dernière égalité vient de~\eqref{euler}.

De la même manière, si $X$ et $Y$ sont indépendants et suivent  la loi Zêta de paramètre $s$, la probabilité que $X$ et $Y$ soient premiers entre eux est la probabilité qu'aucun nombre premier ne divise les deux; par indépendance et avec le lemme~\ref{calculproba}, c'est encore
$\prod_{i\ge 1}(1-\frac1{p_i^{2s}})=\frac1{\zeta(2s)}$.

Lorsque $s$ tend vers $1$, ces probabilités tendent vers $\frac1{\zeta(2)}$, qui est précisément la valeur des limites de $P_n$ et $Q_n$.

\begin{rem}
Cela n'est pas surprenant. En théorie des nombres, on dit qu'une partie $E$ de $\N^*$ admet une densité naturelle, et que cette densité naturelle vaut $\ell$ si
$$\miniop{}{\lim}{n\to +\infty}\P_n(E)=\ell,\text{ où }\P_n(E)=\frac{|\{1,\dots,n\}\cap E|}n.$$
On dit  également qu'une partie $E$ de $\N^*$ admet une densité analytique, ou encore une densité de Dirichlet, et que cette densité vaut $\ell$ si 
$$\miniop{}{\lim}{s \to 1^+}\zeta_s(E)=\ell,$$
où  $\zeta_s$ est la loi Zêta de paramètre $s$.
La notion de densité de Dirichlet, ou densité analytique, est très utilisée en théorie analytique et probabiliste des nombres. On peut par exemple se référer à l'ouvrage de Tenenbaum~\cite{tenenbible}.

En particulier, il est bien connu que l'existence d'une densité naturelle implique celle d'une densité de Dirichlet, et que dans ce cas, les deux densités coïncident.

Ainsi, si $(Q_n)$ converge, ce ne peut être que vers $1/\zeta(2)$.~\footnote{L'exercice 2 page 52 de Garet et Kurtzmann~\cite{gk} utilise ce chemin hybride: on commence par montrer l'existence de la limite de $Q_n$ par la méthode combinatoire classique, puis la limite est identifiée en passant par la densité de Dirichlet.}

\end{rem}

Pour l'heure, nous n'avons pas encore de preuve purement probabiliste, qui ne requière pas une certaine familiarité avec la théorie des nombres.
Cependant, on peut noter que la probabilité limite $\frac{6}{\pi^2}$ est précisément la probabilité qu'une variable suivant la loi Zêta de paramètre 2 soit égale à un. D'où l'idée d'obtenir ce théorème de Dirichlet comme corollaire d'une convergence en loi.

\section{Compléments de probabilité} 
\subsection{Convergence en loi}
La notion de convergence en loi est une notion très importante en théorie des probabilités. Le paragraphe qui suit en fait une présentation minimale, dans le cadre des probabilités sur un ensemble dénombrable. Cette présentation doit permettre de rendre l'essentiel  du texte accessible à un probabiliste débutant -- typiquement un étudiant de deuxième année de licence des Universités ou des classes préparatoires scientifiques.

Soit $\Omega$ un ensemble fini ou dénombrable.
On note $\mathcal{M}(\Omega)$ l'ensemble des probabilités sur $\Omega$, c'est-à-dire
des familles $(\mu(x))_{x\in \Omega}$ avec $\mu(x)\ge 0$ pour tout $x\in \Omega$ et
$\sum_{x\in \Omega} \mu(x)=1$. 

Si $\P\in\mathcal{M}(\Omega)$ et $A\subset \Omega$, on peut définir une fonction $\phi_A$ de $\Omega$ dans  $\{-1,1\}$ par $\phi_A(x)=\1_A(x)-\1_{A^c}(x)$ pour tout $x\in\Omega$. On peut alors écrire
\begin{align*}
2\P(A)-1&=\E[2\1_A-1]=\E[\1_A-\1_{A^c}]=\E[\phi_A(X)]\\&=\sum_{x\in  \Omega}\phi_A(x)\P(x).
\end{align*}
 
L'identification des deux extrémités de l'égalité ne nécessite pas la notion d'espérance, mais elle est alors moins naturelle.\\
Ainsi, pour $\P,\Q\in\mathcal{M}(\Omega)$ et $A\subset \Omega$, on a
$$\P(A)-\Q(A)=\frac12 \sum_{x\in  \Omega}\phi_A(x)\left(\P(x)-\Q(x)\right),$$
d'où $|\P(A)-\Q(A)|\le\frac12 \sum_{x\in  \Omega} |\P(x)-\Q(x)|$.

Notons que l'égalité est atteinte pour $A=\{x\in \Omega;\P(x)\ge \Q(x)\}$.
Ainsi, si l'on pose $d(\P,\Q)=\sup_{A\subset \Omega}|\P(A)-\Q(A)|$, on a $$d(\P,\Q)=\frac12 \sum_{x\in \Omega}|\P(x)-\Q(x)|.$$ Il est facile de voir qu'on a ainsi défini une distance entre les probabilités sur $\Omega$. Cette distance est appelée distance de la variation totale.

En prenant pour $A$ un singleton, on a facilement
$$\sup_{x\in \Omega} |\P(x)-\Q(x)|\le d(\P,\Q).$$

Cependant
\begin{align*}
\frac12 |\P(x)-\Q(x)|&=\frac12(\Q(x)-\P(x))\\&\quad +(\P(x)-\min(\P(x),\Q(x))),
\end{align*} d'où en faisant la somme
$$d(\P,\Q)=\sum_{x\in \Omega}(\P(x)-\min(\P(x),\Q(x))).$$
 
On peut maintenant démontrer le théorème suivant: 
\begin{theo}
Soient $(\P_n)_{n\ge 1},\P$ des probabilités sur $\Omega$. On a équivalence entre
\begin{enumerate}
\item[(i)] $d(\P_n,\P)\to 0$
\item[(ii)] pour tout $A\subset \Omega$, $\P_n(A)\to \P(A)$
\item[(ii')] $\sup_{x\in \Omega} |\P_n(x)-\P(x)|\to 0$
\item[(iii)] pour tout $x\in \Omega$, $\P_n(x)\to \P(x)$
\end{enumerate}
\end{theo}
\begin{proof}
Par définition de $d$, $(i)$ entraîne $(ii)$ et $(ii')$ et il est facile de voir que  $(ii)$ ou $(ii')$ entraîne $(iii)$.
Le plus difficile est de montrer que $(iii)$ entraîne $(i)$, mais nous avons fait le travail préparatoire.
Supposons donc que $\P_n(x)$ tend vers $\P(x)$ pour tout $x$. 
Soit $\epsilon>0$. De la convergence de la somme $\sum_{x\in \Omega}\P(x)$, on déduit qu'il existe $S\subset \Omega$ fini tel que $\sum_{x\in \Omega\backslash S}\P(x)\le \epsilon/2$. 
On a
\begin{align*}d(\P,\P_n)&=\sum_{x\in \Omega}(\P(x)-\min(\P(x),\P_n(x)))\\
&= \sum_{x\in S}(\P(x)-\min(\P(x),\P_n(x)))\\&\quad +\sum_{x\in \Omega\backslash S}(\P(x)-\min(\P(x),\P_n(x)))\\
&\le \sum_{x\in S}(\P(x)-\min(\P(x),\P_n(x)))\\&\quad +\sum_{x\in \Omega\backslash S}\P(x)\\
&\le \sum_{x\in S}(\P(x)-\min(\P(x),\P_n(x)))+\epsilon/2
\end{align*}
Comme la somme est finie, $\miniop{}{\lim}{n\to +\infty}\sum_{x\in S}(\P(x)-\min(\P(x),\P_n(x)))=0$, donc pour $n$ suffisamment grand $\sum_{x\in S}(\P(x)-\min(\P(x),\P_n(x)))\le \epsilon/2$ et donc $d(\P,\P_n)\le\epsilon$.
\end{proof}
Lorsque l'une des conditions équivalentes est vérifiée, on dit que la suite de probabilités $\P_n$ converge en loi vers $\P$.

Si $X,(X_n)_{n\ge 1}$ sont des variables aléatoires sur un espace probabilisé régi par la probabilité $\P$, on dit que $X_n$ converge\footnote{En toute rigueur, il conviendrait de dire que, sous $\P$, $X_n$ converge en loi vers $X$. La précision n'est apportée que dans les rares cas où il y a ambiguïté sur la probabilité sous-jacente.}  en loi vers $X$ si la loi de $X_n$ (notée usuellement $\P_{X_n}$) converge vers la loi de $X$ ($\P_X$).

On peut également noter que

\begin{align*}
\P_{X_n}(A)-\P_{X}(A)&=\P(X_n\in A)-\P(X\in A)\\&=\E[\1_{\{X_n\in A\}}-\1_{\{X\in A\}}],
\end{align*}
donc
\begin{align*}
|\P_{X_n}(A)-\P_{X}(A)|&\le\E[|\1_{\{X_n\in A\}}-\1_{\{X\in A\}}|]\\&\le\E[\1_{\{X_n\ne X\}}]=\P(X_n\ne X),
\end{align*}
d'où $d(\P_{X_n},\P_X)\le \P(X_n\ne X)$.

Cette inégalité simple a des conséquences importantes. 
La plus immédiate est que la convergence des variables aléatoires entraîne la convergence des lois associées.
\begin{theo}
\label{cvpsentraineloi}
Si la suite de variables aléatoires $(X_n)_{n\ge 1}$ est à valeurs dans un ensemble discret et que $(X_n(x))_{n\ge 1}$ tend vers $X(x)$ pour tout $x\in\Omega$, alors  $X_n$ converge en loi vers $X$.
\end{theo}
\begin{proof} Posons $L=\max(n;X_n\ne X)$. La variable $L$ est une variable aléatoire à valeurs dans $\N$, donc $\lim_{n\to +\infty} \P(n\le L)=0$ d'après les propriétés classiques des fonctions de répartition. Comme $d(\P_{X_n},\P_X)\le \P(X_n\ne X)\le \P(n\le L)$,
 $X_n$ converge en loi vers $X$.
\end{proof}

\subsection{Tension}
Si $(\mu_n)_{n\ge 1})$ est une suite à valeurs dans $\mathcal{M}(\Omega)$, le procédé diagonal d'extraction permet d'en extraire une sous-suite  $(\mu_{n_k})_{k\ge 1}$ telle que pour tout $x\in \Omega$,  $(\mu_{n_k}(x))_{k\ge 1}$ converge.
Si l'on pose  alors, pour tout $x\in \Omega$,  $\mu_{\infty}(x)=\lim_{k\to +\infty} \mu_{n_k}(x)$, on a alors, pour toute partie finie $F$ de $\Omega$.

$$\sum_{x\in F} \mu_{\infty}(x)=\lim_{k\to +\infty} \sum_{x\in F} \mu_{n_k}(x)=\lim_{k\to +\infty} \mu_{n_k}(F)\le 1.$$
En passant à la borne supérieure, on obtient $\sum_{x\in \Omega} \mu_{\infty}(x)\le 1$.
Bien sûr, $\mu_{\infty}(x)\ge 0$ pour tout $x\in \Omega$, mais c'est insuffisant pour affirmer que $\mu_{\infty}\in\mathcal{M}(\Omega)$, puisqu'il faudrait encore que  $\sum_{x\in \Omega} \mu_{\infty}(x)=1$. De fait, il peut y avoir une perte de mesure, comme on peut le voir en prenant $\Omega=\N^*$ et $\mu_n=\delta_n$.

Cela amène à définir la notion de tension: on dit qu'une suite à valeurs dans $\mathcal{M}(\Omega)$ est tendue si pour tout $\epsilon>0$, il existe $F$ fini avec
$\mu_n(F)\ge 1-\epsilon$ pour tout $n\ge 1$.

On peut alors énoncer le théorème suivant:
\begin{theo}
De toute suite tendue de $\mathcal{M}(\Omega)$, on peut extraire une sous-suite qui converge en loi.
\end{theo}
\begin{proof}
Reprenons la suite  $(n_k)_{k\ge 1}$ considérée plus haut. Cette fois-ci, pour $\epsilon>0$, on peut trouver $F$ fini tel que $\mu_n(F)\ge 1-\epsilon$ pour tout $n\ge 1$.
On a alors avec les notations précédentes
$$\sum_{x\in F} \mu_{\infty}(x)=\lim_{k\to +\infty} \sum_{x\in F} \mu_{n_k}(x)=\lim_{k\to +\infty} \mu_{n_k}(F)\ge 1-\epsilon,$$
donc $\sum_{x\in \Omega} \mu_{\infty}(x)\ge \sum_{x\in F} \mu_{\infty}(x)\ge 1-\epsilon$.
Comme $\epsilon$ peut être pris arbitrairement petit, $\mu_{\infty}$ est, cette fois, une probabilité.
\end{proof}
Une probabilité limite d'une sous-suite extraite est appelée loi limite de la suite. Comme dans le cas des suites réelles à valeurs dans un compact, on peut énoncer:
\begin{theo}
Si une suite tendue a une unique loi limite, alors elle est convergente.
\end{theo} 
\begin{proof}
Soit $\mu_{\infty}$ l'unique valeur d'adhérence.
Soit $x\in \Omega$. La suite $(\mu_n(x))_{n\ge 1}$ est à valeurs dans $\mathopen[0,1\mathclose]$ qui est compact.
Soit $(n_k)_{k\ge 1}$ une suite telle que $\mu_{n_k}(x)\to \alpha$.
Comme $(\mu_n)_{n\ge 1}$ est tendue, $(\mu_{n_k})_{k\ge 1}$ l'est aussi. On peut donc en extraire une sous-suite  $(\mu_{n_{\phi(k)}})_{k\ge 1}$ qui converge en loi.
Mais cette suite est extraite de $(\mu_n)$ qui n'a qu'une seule valeur d'adhérence, donc $\lim_{k\to +\infty}\mu_{n_{\phi(k)}}(x)=\mu_{\infty}(x)$, ce qui entraîne que $\alpha=\mu_{\infty}(x)$. Ainsi, $\mu_{\infty}(x)$ est l'unique valeur d'adhérence de la suite $(\mu_n(x))_{n\ge 1}$ qui est à valeurs dans un compact, donc  $\mu_n(x)\to \mu_{\infty}(x)$. Comme $\mu_{\infty}\in\mathcal{M}(\Omega)$, $(\mu_n)_{n\ge 1}$ converge en loi vers $\mu_{\infty}$.
\end{proof}

Les notions de convergence en loi, de tension, se généralisent à des lois ou des variables aléatoires sur $\R$ ou $\Rd$. Pour plus d'informations, on pourra par exemple se référer à~\cite{gk}.

\subsection{Lois jointes}

Si $(X_n)_{n\in D}$ est une famille infinie de variables aléatoires,
on appelle lois jointes des $(X_n)_{n\in D}$ les lois des vecteurs aléatoires
$(X_i)_{i\in I}$, où $I$ décrit l'ensemble des parties finies de $D$.
On admettra que la loi d'un vecteur aléatoire $(X_i)_{i\in I}$ est pleinement déterminée par la fonction de répartition inverse:
$$(t_i)_{i\in I}\mapsto\P(\forall i\in I\quad X_i\ge t_i).$$
C'est un fait général dont la démonstration est immédiate dans le cas de la théorie de la mesure. On peut toutefois noter que dans le cas des variables discrètes, cela peut se démontrer par récurrence sur $|I|$ sans difficulté particulière. 
\section{Une preuve pleinement probabiliste}
\label{lasectiondesprobas}
On a besoin d'introduire une notation et quelques lemmes: d'abord, on note $\nu_p(n)$ l'exposant de $p$  dans la décomposition de $n$ en produit de facteurs premiers (c'est la valuation $p$-adique de $n$).

\begin{lem}
\label{reconnais_lois_un}
La loi d'une variable aléatoire $X$ à valeurs dans $\N^*$ est caractérisée par les lois jointes des variables $(\nu_{p_k}(X))_{k\ge 1}$.
\end{lem}
\begin{proof}
On a $$X=\lim_{n\to +\infty} \psi_n(X),\text{ avec }\psi_n(x)=\prod_{i=1}^n p_i^{\nu_{p_i}(x)}.$$
La suite $(\psi_n(X))_{n\ge 1}$ converge ponctuellement vers $X$, donc converge en loi vers $X$, d'après le théorème~\ref{cvpsentraineloi}.
Si les lois jointes des variables $(\nu_{p_k}(X))_{k\ge 1}$ sont connues, les lois des $(\psi_n(X))_{n\ge 1}$ sont connues, et donc la loi de $X$.
\end{proof}

\begin{lem}
La loi d'une variable aléatoire à valeurs dans $\N^*$ est caractérisée par les valeurs de $\P(n|X)$, où $n$ décrit $\N^*$.
\end{lem} 
\begin{proof}
D'après le lemme précédent, il suffit de connaître les loi jointes des variables $(\nu_{p_k}(X))_{k\ge 1}$. Mais pour cela, il suffit d'avoir les fonctions de répartition inverse. Or,
$$\P(\nu_{p_1}(X)\ge a_1,\dots,\nu_{p_r}(X)\ge a_r)=\P(\prod_{i=1}^r p_i^{a_i}| X),$$
d'où le résultat.
\end{proof}

Ainsi, si $X$ et $Y$ sont des variables indépendantes suivant respectivement les lois Zêta de paramètre $s$ et $t$, alors leur plus grand commun diviseur, noté $X\wedge Y$, vérifie $\P(n|X\wedge Y)=\P(n|X,n|Y)=\P(n|X)\P(n|Y)=\frac1{n^{s+t}}$, donc $X\wedge Y$ suit la loi Zêta de paramètre $s+t$.

Dans le cas où  $X$ suit la loi Zêta de paramètre $s$, le calcul de fonction de répartition effectué dans la preuve du lemme nous donne 
$$\P(\nu_{p_1}(X)\ge a_1,\dots,\nu_{p_r}(X)\ge a_r)=\prod_{i=1}^r  p_i^{-a_i s}.$$
La fonction de répartition inverse caractérisant la loi des vecteurs, on en déduit que les variables $\nu_p(X)$ sont indépendantes et que $1+\nu_p(X)$ suit la loi géométrique de paramètre $1-p^{-s}$ (que l'on note $\mathcal{G}(1-p^{-s})$). 

\begin{theo}
Soit $X,(X_n)_{n\ge 1}$ des variables aléatoires à valeurs dans $\N$.
On suppose que
\begin{itemize}
\item $(X_n)$ est tendue.
\item Pour tout $N\ge 1$, $\P(N|X_n)\to \P(N|X)$
\end{itemize}
Alors $(X_n)_{n\ge 1}$ converge en loi vers $X$.
\end{theo}

\begin{proof}
Comme la famille est tendue, il suffit d'identifier les lois limites:
si $X_{n_k}$ tend en loi vers $Y$, $\P(X_{n_k}\in n\N^*)$ tend vers $\P(Y\in n\N^*)$; autrement dit $\P(n|X_{n_k})$ tend vers $\P(n|Y)$. Donc  $\P(n|Y)=\P(n|X)$ et $X$ et $Y$ ont même loi.
\end{proof}

On peut maintenant énoncer et démontrer un résultat probabiliste.

\begin{theo}
\label{dansz}
Soient $X_n$, $Y_n$ des variables aléatoires suivant la  loi uniforme sur $\{1,\dots,n\}$.
On note $Z_n=X_n\wedge Y_n$ et $W_n=r(X_n)$, où $r(n)$ est le plus grand entier $a$ tel que $a^2$ divise $n$.

Alors $W_n$ et $Z_n$ convergent en loi vers la loi Zêta de paramètre $2$.
\end{theo}
\begin{proof}
On a
$$\P(N|Z_n)=\P(N|X_n,N|Y_n)=\P(N|X_n)^2=\left(\frac{\floor{n/N}}{n}\right)^2$$
et 
$$\P(N|W_n)=\P(N^2|X_n)=\frac{\floor{n/N^2}}{n}.$$
Les deux quantités convergent évidemment vers $\frac1{N^2}$, qui est la probabilité qu'une variable suivant la loi Zêta de paramètre 2 soit divisible par $N$.

Reste à montrer la tension. Quels que soient les entiers naturels $n$ et $a$, on a
$$\P(Z_n\ge a)=\sum_{i=a}^{+\infty}\P(Z_n=i)\le \sum_{i=a}^{+\infty}\P(i| Z_n)\le \sum_{i=a}^{+\infty}\frac1{i^2},$$
et de même $\P(W_n\ge a)\le \sum_{i=a}^{+\infty}\frac1{i^2}$.
Soit alors $\epsilon>0$. Si je prends $a$ tel que $\sum_{i=a}^{+\infty}\frac1{i^2}<\epsilon$, l'ensemble $F=\{1,\dots,a\}$ vérifie $\P(Z_n\in F)\ge 1-\epsilon$ et 
$\P(W_n\in F)\ge 1-\epsilon$, ce qui montre que l'hypothèse de tension est bien satisfaite.
 
\end{proof}

Ce résultat se généralise aisément:
\begin{theo}
Soient $X^1_n,\dots,X^m_n$ des variables aléatoires indépendantes suivant la  loi uniforme sur $\{1,\dots,n\}$.
On note $Z_n=X^1_n\wedge \dots \wedge X^m_n$ et $W_n=r_m(X^1_n)$, où $r_m(n)$ est le plus grand entier $a$ tel que $a^m$ divise $n$.

Alors $W_n$ et $Z_n$ convergent en loi vers la même loi Zêta de paramètre $m$.
\end{theo}
La preuve est laissée au lecteur.
On a ainsi retrouvé un ancien résultat, que son découvreur, Ernest Cesàro~\cite{cesaro1,cesaro2}, décrivait en ces termes (voir~\cite{cesaro2}) : \og La probabilité que la $m$-ième racine de la plus haute puissance  $m$-ième, qui divise un nombre entier pris au hasard, appartienne à un certain système de nombres, ne diffère pas de la probabilité que le plus grand commun diviseur de $m$ entiers, pris au hasard, appartienne au même système.\fg

\section{Quelques généralisations}

Dans cette section, on propose quelques généralisations/extensions des résultats précédents, toujours en privilégiant, lorsque cela est possible, l'approche probabiliste. On supposera maintenant que les théorèmes de convergence dominée et de convergence monotone sont bien connus, particulièrement lorsqu'on les applique à des variables aléatoires.

\subsection{Développement eulérien}
Une fonction $\phi:\N^*\to\C$ est dit multiplicative si $\phi(pq)=\phi(p)\phi(q)$ est vérifiée dès que $p$ et $q$ sont premiers entre eux.
Si c'est vrai pour tous les couples $(p,q)\in(\N^*)^2$, alors la fonction est dite complètement multiplicative.

\begin{theo}
Soit $\phi$ une fonction multiplicative positive ou bornée.
On suppose que $X$ suit la loi  Zêta de paramètre $s$
et que $(N_i)_{i\ge 1}$ est une suite de variables aléatoires indépendantes telles que $1+N_i\sim\mathcal{G}(1-p_i^{-s})$.
Alors
\begin{align}
\label{decomp_proba}
\E[\phi(X)]&=\prod_{i=1}^{+\infty}\E(p_i^{N_i})\end{align}
et
\begin{align}
\label{decomp_premier}
\sum_{n=1}^{+\infty} \frac{\phi(n)}{n^s}=\prod_{i=1}^{+\infty}\left(\sum_{j=0}^{+\infty}p_i^{-sj}\phi(p_i^j)\right).
\end{align}
En particulier, si $\phi$ est complètement multiplicative
\begin{align}
\label{decomp_complete}
\sum_{n=1}^{+\infty} \frac{\phi(n)}{n^s}&=\prod_{i=1}^{+\infty} (1-p_i^{-s}\phi(p_i))^{-1}
\end{align}

\end{theo}

\begin{proof}
Soient $(N_i)_{n\ge 1}$ une suite de variables aléatoires indépendantes telles que $1+N_i\sim\mathcal{G}(1-p_i^{-s})$. Supposons $\phi$ multiplicative et bornée.
D'après le lemme~\ref{reconnais_lois_un}, $\prod_{i=1}^{+\infty} p_i^{N_i}$ a même loi que $X$, donc
$\E[\phi(X)]=\E[\phi(\prod_{i=1}^{+\infty} p_i^{N_i})]$; mais $\phi$ est une fonction multiplicative, donc
$$\E[\phi(X)]=\E[\prod_{i=1}^{+\infty} \phi(p_i^{N_i})]=\prod_{i=1}^{+\infty}\E(\phi(p_i^{N_i})).$$
Comme $\phi$ est bornée, l'égalité entre le deuxième et le troisième membre  est une conséquence du théorème de convergence dominée et de l'indépendance des
 variables $(\phi(p_i^{N_i})_{i\ge 1}$.
Ainsi
$$\zeta(s)^{-1}\sum_{n=1}^{+\infty}\frac{\phi(n)}{n^s}=\prod_{i=1}^{+\infty}\left(\sum_{j=0}^{+\infty}(1-p_i^{-s})p_i^{-sj}\phi(p_i^j)\right).$$
En simplifiant par $\zeta(s)^{-1}=\prod_{i=1}^{+\infty}(1-p_i^{-s})$, on obtient le résultat voulu.

Le cas où $\phi$ est positive non bornée se fait en l'approchant par une suite croissante de fonctions multiplicatives bornées. Ceci est laissé en exercice au lecteur.
\end{proof}

\subsection{Nombre moyen de décompositions en sommes de carrés}
On note $\mathbb{Z}[i]$ l'ensemble des nombres de la forme $a+ib$, avec $a$ et $b$ dans $\Z$. C'est l'anneau des entiers de Gauss. Pour $z\in \mathbb{Z}[i]$, on pose $N(z)=z\overline{z}$.
On a bien sûr $N(zz')=N(z)N(z')$.
Ainsi, un inversible de $\mathbb{Z}[i]$ est tel que $N(z)$ est un inversible de $\Z$: ce ne peut être que $1,i,-1,-i$. Il est alors aisé de voir que tout élément de $\Z[i]$ non nul s'écrit d'une manière unique sous la forme $i^k z'$, avec 
$k\in\{0,1,2,3\}$, $\Re z'\ge 0$ et $\Im z'>0$: $z'$ est le représentant privilégié de la classe de $z$ lorsque l'on quotiente le semi-groupe $(\Z[i],\times)$ par ses éléments inversibles. On a bien sûr $N(z)=N(z')$, de sorte que l'application $N$ passe au quotient.
Ainsi, si on note $Z'$ l'ensemble des classes non nulles,
on peut définir une application $S'$ de $Z'$ dans $\N$ par
$$S'(n)=|\{z\in Z'; N(z)=n\},$$
et l'on a également $S'(n)=|\{(a,b)\in\N\times\N^*; a^2+b^2=n\}|$.

La fonction $S'$ peut se calculer explicitement, en utilisant un certain nombre de résultats bien connus de l'anneau $\Z[i]$, qui sont par exemple décrits dans Perrin~\cite{perrin}.
On sait en particulier que
\begin{itemize}
\item $\Z[i]$ est un anneau factoriel.
\item Les irréductibles de $\Z[i]$ sont 
\begin{itemize}
\item les nombres premiers congrus à $3$ modulo 4; ces nombres ne peuvent s'écrire sous forme de somme de deux carrés.
\item les nombres de la forme $a+ib$ tels que $a^2+b^2$ est un nombre premier; tous les entiers naturels premiers qui ne sont pas congrus à $3$ modulo 4 peuvent s'écrire sous la forme d'une somme de deux carrés.
\end{itemize}
\end{itemize}
Ces résultats étant rappelés, on commence par établir un lemme très utile:
\begin{lem}
Si $a$ et $b$ sont des entiers naturels premiers entre eux, tout élément de $Z'$ de norme $ab$ se factorise de manière unique sous la forme du produit de deux éléments de $Z'$ de normes respectives $a$ et $b$.
\end{lem}
\begin{proof}
Soit $z'\in Z'$ avec $N(z')=ab$.
$z'$ se factorise dans $Z'$ comme produits de classes d'éléments irréductibles de $\Z[i]$. La norme d'un facteur divise $N(z')=ab$, donc soit $a$, soit $b$ puisque $a$ et $b$ sont premiers entre eux.
Soit $x$ le produit des facteurs (pris avec leur multiplicité) dont la norme divise $a$, $y$ le produit des facteurs dont la norme divise $b$. On a $z'=xy$. $N(x)$ divise $N(z')=ab$ et est premier avec $b$ donc $N(x)$ divise $a$. De même $N(y)$ divise $b$. Comme leur produit fait $ab$, on a $N(x)=a$ et $N(y)=b$.
Maintenant si $z=x'y'$ avec $N(x)=a$ et $N(x)=b$, un facteur irréductible de  $x'$ est un facteur irréductible de $z$ dont la norme divise $a$: c'est un facteur irréductible de $x$. La valuation de $q$ dans $x'$ ne peut être plus grande que dans $z$, puisque $x'$ divise $z$. Or, par définition de $x$ la valuation de $q$ dans $x$ est égale à la valuation de $q$ dans $z$, donc $q$ a une valuation plus petite dans $x'$ que dans $x$, ce pour tout $q$, donc $x'$ divise $x$.
Comme $x'$ et $x$ ont la même norme, ils sont égaux.
De même, $y=y'$.
\end{proof}
On en déduit directement que la fonction $S'$ est une fonction multiplicative.
Calculons plus précisément cette fonction.

Si $z'\in Z'$ est tel que $N(z')=2^e$, un facteur irréductible de $z'$ a une norme qui est une puissance de deux. Vu la caractérisation des irréductibles rappelée plus haut, ce facteur ne peut être que la classe de $1+i$. Finalement, la
classe de $(1+i)^{e}$ est la seule classe de norme $2^e$.

Si $p$ est un nombre premier congru à $1$ modulo 4, $p$ n'est pas premier dans  $\Z[i]$; il se ramifie sous la forme $p=(a+ib)(a-ib)$, $a+ib$ et $a-ib$ sont premiers dans $\Z[i]$, non équivalents, de norme $p$.
Si $N(z)=p^e$, un facteur premier de $z$ a une norme qui divise $p^e$, donc $p$: les facteurs de $z$ ne peuvent être que les classes de  $a+ib$ et $a-ib$.
Ainsi $z$ s'écrit comme une puissance $k$-ième de la classe de $a+ib$ et
une puissance $\ell$-ième de la classe de $a-ib$. Comme $p^e=N(z)=N(a+ib)^{k}N(a-ib)^{\ell}=p^{k+\ell}$, cela nous donne exactement $e+1$ solutions.

Si $p$ est un nombre premier congru à $3$ modulo 4, $p$ est premier dans 
$\Z[i]$. Si $N(z)=p^e$, un facteur premier de $z$ a une norme qui divise $p^e$, donc $p$: ce ne peut être que $p$, donc $z$ s'écrit $z=p^{k}$, et on a $N(z)=p^{2k}$. Il n'y a donc de solution que si $e$ est pair, et dans ce cas, elle est unique.

Ainsi, on a démontré que la fonction $S'$ est une fonction multiplicative, que l'on peut calculer explicitement avec pour $p$ premier:
$$S'(p^e)=\begin{cases}1&\text{si }p=2\\
e+1&\text{si $p$ est congru à 1 modulo 4}\\ 
\1_{\{e\text{ pair}\}}&\text{si $p$ est congru à 3 modulo 4}
\end{cases}$$

\begin{theo}
Soit $X$ suivant la loi Zêta de paramètre $s>1$. On a $\E[S'(X)]=\beta(s)$, où $\beta$ est la fonction bêta de Dirichlet:
$$\beta(s)=\sum_{n=0}^{+\infty}\frac{(-1)^n}{(2n+1)^s}.$$
\end{theo}
\begin{proof}
Soient $(N_i)_{i\ge 1}$ une suite de variables aléatoires indépendantes telles que $1+N_i\sim\mathcal{G}(1-p_i^{-s})$.
On a $\E[S'(X)]=\prod_{i=1}^{+\infty} \E(S'(p_i^{N_i}))]$
Ainsi, si $p_i$ est congru à 1 modulo 4, on a $\E(S'(p_i^{N_i}))=\E[N_i+1]=\frac1{1-p_i^{-s}}$ et si $p_i$ est congru à 3 modulo 4, 
$\E(p_i^{N_i})=\sum_{j=0}^{+\infty} (1-p_i^{-s})p_i^{-2js}=\frac{1-p_i^{-s}}{1-p_i^{-2s}}=\frac1{1+p_i^{-s}}$.
Ainsi, dans tous les cas $\E(p_i^{N_i})=(1-p_i^{-s}\chi_4(p_i))^{-1}$ où $\chi_4$ est
défini par $\chi_4(2n)=0$ et  $\chi_4(2n+1)=(-1)^n$.
Finalement, comme $\chi_4$ est complètement multiplicative, on a
\begin{align*}
\E[S'(X)]&=\prod_{i=1}^{+\infty}\E(S'(p_i^{N_i}))\quad\text{ avec~\eqref{decomp_proba}}\\
&=\prod_{i=1}^{+\infty}(1-p_i^{-s}\chi_4(p_i))^{-1}\\
&=\sum_{n=1}^{+\infty}\frac{\chi_4(n)}{n^s}\quad\text{ avec~\eqref{decomp_complete}}\\
&=\sum_{n=0}^{+\infty}\frac{(-1)^n}{(2n+1)^s}=\beta(s)
\end{align*}
\end{proof}

\subsection{Application aux entiers de Gauss}
Comme on l'a fait pour les nombres entiers, on veut maintenant mesurer la propension qu'ont les entiers de Gauss à être premiers entre eux.\\

Soit $s>1$. En identifiant $Z'$ avec $\{a+ib; (a,b)\in\N\times\N^*,(a,b)\ne (0,0)\}$, on voit que $Z'$ s'identifie à une partie dénombrable discrète de $\R^2$.
En regroupant les termes suivant la valeur de $N(z)$, 
 on voit que la quantité $\sum_{z\in Z'}N(z)^{-s}$ peut se réécrire
$$\sum_{z\in Z'}N(z)^{-s}=\sum_{n=1}^{+\infty} \frac{S'(n)}{n^s},$$
où $S'$ est la fonction étudiée à la sous-section précédente.
On en déduit que $\sum_{z\in Z'}N(z)^{-s}=\zeta(s)\beta(s)$.
En particulier, la série de terme général  $(N(z)^{-s})_{z\in Z'}$ converge, et on peut alors définir une loi $\zeta'_s$ sur $Z'$  par
$$\zeta'_s(x)=(\sum_{z\in Z'}N(z)^{-s})^{-1} N(x)^{-s}$$
Il est alors aisé de voir que pour tout $z\in Z$ $\zeta'_s(zZ)=\frac1{N(z)^s}$.

Comme $\Z$, $\Z[i]$ est un anneau factoriel, et on a unicité de la décomposition des éléments de $Z'$ comme produit de classes premières. 

En procédant comme dans la section~\ref{lasectiondesprobas}, on montre alors
\begin{theo}
Soit $X,(X_n)_{n\ge 1}$ des variables aléatoires à valeurs dans $Z'$.
On suppose que
\begin{itemize}
\item $(X_n)$ est tendue.
\item Pour tout $z\in Z'$, $\P(z|X_n)\to \P(z|X)$
\end{itemize}
Alors $(X_n)_{n\ge 1}$ converge en loi vers $X$.
\end{theo}
\begin{proof}
Comme $\Z[i]$ est discret et dénombrable, les arguments de la section~\ref{lasectiondesprobas} se déroulent sans grande modification. Les détails de la preuve sont laissés au lecteur. 
\end{proof}
On peut maintenant énoncer un analogue dans $\Z[i]$ du théorème~\ref{dansz} .
\begin{theo}
Soient $X_n$, $Y_n$ des variables aléatoires indépendantes suivant la  loi uniforme sur $\{z\in Z'; N(z)\le n^2\}$.
On note $Z_n=X_n\wedge Y_n$. Alors $Z_n$ converge en loi vers $\zeta'_2$
En particulier
$$\lim_{n\to +\infty}\P(X_n\wedge Y_n=1)=(\sum_{z\in Z'}N(z)^{-2})^{-1}=\frac1{\zeta(2)\beta(2)}.$$
\end{theo}
\begin{proof}
Pour $a\in\C$ et $r\ge 0$, on note $B(a,r)=\{z\in\C\quad |a-z|\le r\}$.
En identifiant un point de $Z'$ avec les $4$ points de $\Z[i]$ qu'il contient, on a
\begin{align*}\P(z|Z_n)=\P(z|X_n)^2&=\left(\frac{|(zZ')\cap B(0,n)|}{|Z'\cap B(0,n)|}\right)^2\\
&=\left(\frac{|Z'\cap B(0,n/|z|)|}{|Z'\cap B(0,n)|}\right)^2\\
&\sim \left(\frac{\pi (n/|z|)^2}{\pi n^2}\right)^2\sim\frac1{N(z)^2}
\end{align*}
Ainsi la loi $\zeta'_2$ est la seule loi limite possible.
Montrons donc la tension.
On peut trouver $a,b>0$ tels que  $ar^2\le |Z'\cap B(0,r)|$ pour $r\ge 1$ et  $|Z'\cap B(0,r)|\le br^2$ pour tout $r>0$.
On en déduit
\begin{align*}\P(Z_n=z)\le \P(z|Z_n)\le \frac{b}{a}\frac1{N(z)^2}
\end{align*}
Comme la série des $\frac1{N(z)^2}$ converge, cela donne comme précédemment la tension de $(Z_n)$.
\end{proof}

On a ainsi obtenu une preuve assez élémentaire d'un résultat obtenu précédemment par Collins et Johnson~\cite{MR1034735} à l'aide de la théorie générale des corps de nombres.

\def\refname{Références}
\bibliographystyle{plain}

\begin{thebibliography}{}

\end{thebibliography}


\begin{thebibliography}{1}

\bibitem{borde}
Olivier Bordellès.
\newblock {\em Thèmes d'arithmétique}.
\newblock Ellipses, 2006.

\bibitem{cesaro1}
Ernest Ces\`{a}ro.
\newblock \'{E}tude moyenne du plus grand commun diviseur de deux nombres.
\newblock {\em Annali di Matematica Pura ed Applicata}, (13):235--250, 1885.

\bibitem{cesaro2}
Ernest Ces\`{a}ro.
\newblock Sur le plus grand commun diviseur de plusieurs nombres.
\newblock {\em Annali di Matematica Pura ed Applicata}, (13):291--294, 1885.

\bibitem{MR1034735}
George~E. Collins and Jeremy~R. Johnson.
\newblock The probability of relative primality of {G}aussian integers.
\newblock In {\em Symbolic and algebraic computation ({R}ome, 1988)}, volume
  358 of {\em Lecture Notes in Comput. Sci.}, pages 252--258. Springer, Berlin,
  1989.

\bibitem{faraut}
Jacques Faraut and Elisabeth Khalili.
\newblock {\em Arithmétique. Cours, Exercices et Travaux Pratiques sur
  Micro-Ordinateur}.
\newblock Ellipses, 1990.

\bibitem{gk}
Olivier Garet and Aline Kurtzmann.
\newblock {\em De l'intégration aux probabilités}.
\newblock Ellipses, 2011.

\bibitem{perrin}
Daniel Perrin.
\newblock {\em Cours d'algèbre}.
\newblock Ellipses, 1996.

\bibitem{tenenbible}
G{\'e}rald Tenenbaum.
\newblock {\em Introduction \`a la th\'eorie analytique et probabiliste des
  nombres}.
\newblock Belin, 2008.

\bibitem{tissier}
Alain Tissier.
\newblock {\em Mathématiques générales à l'usage des candidats à
  l'Agrégation interne de Mathématiques}.
\newblock Bréal, 1991.

\end{thebibliography}

\end{document}